\documentclass[12pt]{amsart}
\usepackage{amsfonts, amssymb, latexsym, tikz}

\newtheorem{thm}{Theorem}[section]
\newtheorem{cor}[thm]{Corollary}

\newtheorem{example}[thm]{Example}
\newtheorem{proposition}[thm]{Proposition}

\theoremstyle{remark}

\theoremstyle{definition}

\title{Bruhat graphs and pattern avoidance}
\author{Christopher Conklin}
\address{School of Physics \& Astronomy, University of Minnesota, 116 Church St. SE, Minneapolis, MN 55455}
\email{conklin@physics.umn.edu}

\author{Alexander Woo}
\address{Department of Mathematics, University of Idaho, P.O. Box 441103,
Moscow, ID 83844-1103}
\email{awoo@uidaho.edu}


\date{June 4, 2013}

\begin{document}
\begin{abstract}
We characterize permutations whose Bruhat graphs can be drawn in the
plane and those whose Bruhat graphs can be drawn in the torus.  In particular, we show these properties are characterized by
avoiding finitely many permutations.
\end{abstract}

\maketitle

\section{Introduction}
The set of all permutations (of an arbitrary finite number of
elements) admits a partial order known as pattern containment order.
This partial order is known to admit infinite antichains~\cite{SpiBon,
  Pratt}.  On the other hand, in almost all cases where the set of
permutations satisfying some property has been characterized by
pattern containment, the number of permutations involved is finite.
For some properties, such as in~\cite{Ten11, BilWeed, Jones, BilPaw}, the
number of permutations is moderate to quite large, so this phenomenon
seems to involve more than merely the natural inclination of
mathematicians to study simpler examples.  Our goal in this paper is
to begin the exploration of one possible explanation for this
finiteness.

Associated to each permutation is a directed graph known as the Bruhat
graph.  (For definitions see Section 2.)  Whenever a permutation $\pi$
is contained in a permutation $\sigma$ (so $\pi\leq\sigma$ in pattern
containment order), the Bruhat graph of $\pi$ is a subgraph of the
Bruhat graph of $\sigma$.  While not all properties characterized by
pattern containment can be reduced to properties of Bruhat graphs,
many of the properties that have been so characterized, especially
those coming from algebraic geometry or representation theory, have a
graph theoretic description.  For example, the permutations $w$
avoiding $3412$ and $4231$ are exactly the ones whose Bruhat graphs
are regular, meaning that they have the same number of edges at each
vertex; these are also the ones associated to smooth Schubert
varieties~\cite{LakSand, Car}.

In this paper we study the permutations whose Bruhat graphs can be drawn
(without crossings) on the plane or on the torus.  We show that the planar permutations are precisely the ones that avoid $321$ and have length at most 3, and their Bruhat graphs are either a single point, a single edge, (the edge graph of) a square, or a cube.  We give an analogous characterization of permutations whose Bruhat graphs can be drawn on the torus; the additional possible graphs are $K_{3,3}$ (the complete bipartite graph on two sets of three vertices) and the $4$-cube.  By an old theorem of Atkinson~\cite{Atk-fin}, our characterizations imply that these properties are characterized by avoidance of finitely many patterns.  A computer calculation shows that 29 patterns are needed to characterize planar permutations, and 92 are needed to characterize toroidal ones.

The motivation for considering Bruhat graphs is the Graph Minor
Theorem, a deep and surprising result of Robertson and
Seymour~\cite{RobSey} stating that graph minor order, the partial
order on graphs generated by deleting and contracting edges, has no
infinite antichains.  In particular, the graphs satisfying any
property which is preserved under deletion and contraction are
characterized by avoiding finitely many graphs.  An important special
case of the Graph Minor Theorem is the one involving the class of
graphs which can be drawn without crossings on a surface of genus
$g$~\cite{RobSeyGenus}, and the original inspiration for graph minor
theory was Kuratowski's Theorem, which characterizes graphs that cannot be
drawn in the plane.  Our results in this paper can therefore be seen
as potential first steps towards using graph minor theory to establish
either directly or by analogy a possible finiteness result for pattern
avoidance under yet unknown hypotheses involving Bruhat graphs.

In particular, we also expect that the set of permutations whose Bruhat graphs can be drawn on a surface of genus $g$ can be characterized by avoidance of a finite number of patterns.  However, it appears that any characterization of these permutations, even if not directly in terms of the patterns to be avoided, will be fairly long even for the case $g=2$.

We were particularly inspired by the question of Billey and
Weed~\cite{BilWeed} concerning whether, for a fixed integer $k$,
having Kazhdan--Lusztig polynomial $P_{id,\sigma}(1)\leq k$ is
characterized by the permutation $\sigma$ avoiding some finite list of
patterns.  (Billey and Braden previously showed in~\cite{BilBra} that
this property is characterized by avoiding a possibly infinite set of
patterns.)  This Kazhdan--Lusztig polynomial is known to be a property
depending only on the Bruhat graph of $\sigma$~\cite{duCloux, Bre,
  Del}, although the aforementioned property is not preserved by
deletion and contraction on Bruhat graphs (nor is its negation).

Our work is in some sense orthogonal to earlier work of Atkinson, Murphy, and
Ruskuc~\cite{Atk-fin,AMR02, AMR06} on finitely generated order ideals in pattern
containment order.  Their work takes the viewpoint that a permutation is a
string consisting of distinct integers.  Our viewpoint considers permutations
as elements of a Coxeter group.  Therefore our results are complementary to
theirs.  Furthermore, pattern avoidance has alternative definitions within the
framework of Coxeter groups~\cite{BilPos}.  We believe our work can be easily
extended to the other Coxeter groups.

Section 2 is devoted to definitions, while Sections 3 and 4
respectively give proofs for our theorems about permutations whose Bruhat graphs are planar and toroidal.

Most of this work was the result of a summer undergraduate research
project at St. Olaf College mentored by the second author.  We
thank St. Olaf College for the financial support that made this
work possible.  The first part of the proof of Theorem~\ref{thm:toroidal}, showing that
the Bruhat graph of 3412 is not toroidal, is due to Michael Eldredge.  We
thank him for allowing us to include it in this paper.

\section{Definitions}

We begin with definitions from the combinatorics of Coxeter groups
applied to the specific case of the symmetric group $S_n$; a standard
reference for this material is~\cite{BjoBre}.

By a {\bf transposition} $t$ we mean some 2-cycle $(i\,\,j)$ in the
symmetric group.  An {\bf adjacent transposition} is one of the form
$(i\,\,i+1)$.

Let $\pi\in S_n$ be a permutation.  The {\bf length} of $\pi$, denoted
$\ell(\pi)$, is the minimum number of adjacent transpositions
$s_{i_1},\ldots,s_{i_\ell}$ such that $\pi$ can be written as their
product $s_{i_1}\cdots s_{i_\ell}$.  An {\bf inversion} of $\pi$ is a
pair of indices $i<j$ with $\pi(i)>\pi(j)$; note that $\ell(\pi)$ is
also equal to the number of inversions of $\pi$.

The {\bf absolute length} of $\pi$, denoted $a(\pi)$, is the minimum
number of transpositions $t_{i_1},\ldots, t_{i_a}$, not necessarily
adjacent transpositions, such that $\pi=t_{i_1}\cdots t_{i_a}$.  If
$\pi\in S_n$ has $c$ disjoint cycles (counting fixed points as
1-cycles), then $a(\pi)=n-c$.  By definition, $a(\pi)\leq \ell(\pi)$
for any permutation $\pi$.

\begin{example}
Permutations are written in one line notation unless stated otherwise.
The permutation $\pi=3412$ has length $\ell(\pi)=4$, and absolute
length $a(\pi)=2$.
\end{example}

The symmetric group $S_n$ has a partial ordering known as {\bf Bruhat
  order}.  Define $\pi\prec\sigma$ if $\ell(\pi)<\ell(\sigma)$ and
there is a transposition $t$ such that $t\pi=\sigma$ (or equivalently
a transposition $t^\prime$ such that $\pi t^\prime=\sigma$).  Bruhat
order is the transitive closure of $\prec$, so $\pi\leq\tau$ if there
exist permutations $\sigma_1,\ldots,\sigma_k\in S_n$ such that
$\pi\prec\sigma_1\prec\cdots\prec\sigma_k\prec\tau$.

The {\bf Bruhat graph} for $S_n$ is the directed graph whose vertices
are the elements of $S_n$, with edges defined as follows.  Given
permutations $\sigma$ and $\tau$, there is an edge
$\sigma\rightarrow\tau$ if $\sigma\prec\tau$, meaning that
$\ell(\sigma)<\ell(\tau)$ and there exists a transposition $t$ with
$\sigma=\tau t$.  (Note $\ell(\tau)-\ell(\sigma)$ need not equal 1.)
Given a permutation $\pi$, the {\bf Bruhat graph for $\pi$}, denoted
$\mathcal{B}(\pi)$, is the induced subgraph whose vertices are those
labeled by permutations $\sigma$ with $\sigma\leq\pi$; this is the
largest subgraph with unique sink $\pi$.

The length of $\pi$ is the length of the longest directed path
(necessarily from the identity to $\pi$) in $\mathcal{B}(\pi)$, and
the absolute length of $\pi$ is the length of the shortest directed
path in $\mathcal{B}(\pi)$ from the identity (which is graph
theoretically the unique source) to $\pi$ (the unique sink).  (Sometimes absolute length is defined as the length of the shortest undirected path, but a theorem of Dyer~\cite{Dyer} shows these two definitions are equivalent.)

Now we give various definitions related to the notion of {\bf pattern
  containment}; a standard reference for this subject is~\cite{Bona}.
Let $\pi\in S_k$ and $\tau\in S_n$ with $k\leq n$.  We say that $\tau$
{\bf (pattern) contains} $\pi$ if there exist indices $1\leq
i_1<\ldots<i_k\leq n$ such that $\tau(i_a)<\tau(i_b)$ if and only if
$\pi(a)<\pi(b)$.  We say $\tau$ {\bf (pattern) avoids} $\pi$ if $\tau$
does not contain $\pi$.

\begin{example}
The permutation $5736241$ contains the permutation $3412$ (both
written in one-line notation) three different ways using the bolded
entries: $\mathbf{573}62\mathbf{4}1$, $\mathbf{57}36\mathbf{24}1$, and
$\mathbf{5}73\mathbf{624}1$.

In contrast, $135246$ avoids $3412$.
\end{example}

The following proposition is immediate.

\begin{proposition}
\label{patterntograph}
If $\tau$ contains $\pi$, then $B(\pi)$ is isomorphic to a subgraph of
$B(\tau)$ which includes the sink vertex $\tau$.
\end{proposition}
\begin{proof}
Let $i_1<\cdots<i_k$ be the indices of by which $\pi$ is contained in
$\tau$.  Consider the induced subgraph given by the vertices of
$B(\tau)$ labelled by permutations $\sigma$ for which
$\sigma(j)=\tau(j)$ for all $j$ not among the containment indices,
meaning that $j\neq i_a$ for all $a$.
\end{proof}

Pattern containment is a partial order relation; the poset of
permutations under pattern containment is sometimes called {\bf
  pattern order}.

Finally, we need several definitions from graph theory.  A graph is {\bf
  planar} if it can be drawn in the plane without edges crossing, and a graph is {\bf toroidal} if it can be drawn on the surface of a torus without edges crossing.  The
following theorem is classical and known as Kuratowski's Theorem.

\begin{thm}
A graph is planar if and only if it contains no subgraph isomorphic to
a subdivision of $K_{3,3}$ (the complete bipartite graph with 3
vertices on each side) or $K_5$ (the complete graph on 5 vertices).
\end{thm}

The analogous theorem for toroidal graphs is not known; it is known that the list of graphs to avoid is at least thousands of graphs long.

Finally, given some subset $W$ of the set of vertices of a graph $G$, the {\bf subgraph induced by $W$} is the graph containing the vertices in $W$ and every edge of $G$ which connects two vertices of $W$.  Any subgraph $H\subset G$ that can be formed in this manner (or, in other words, any subgraph $H$ where every edge of $G$ between two vertices of $H$ is also in $H$) is called an {\bf induced subgraph}.

\section{Planar Bruhat graphs}

In this section we give a number of characterizations of permutations
with planar Bruhat graphs, beginning with the following.

\begin{thm}
\label{thm:planar}
Let $\sigma$ be a permutation.  Then the Bruhat graph $\mathcal{B}(\sigma)$ is
planar if and only if $\sigma$ avoids 321 and $\ell(\sigma)\leq 3$.
\end{thm}

\begin{proof}
The Bruhat graph $\mathcal{B}(321)$ is the complete bipartite graph
$K_{3,3}$, which is not planar.  Furthermore, $\mathcal{B}(3412)$
contains (as a subgraph) $\mathcal{B}(1432)$ since $1432<3412$ in
Bruhat order, and $1432$ contains $321$, so $\mathcal{B}(3412)$ is
also not planar.  Therefore, if $\sigma$ contains $321$ or $3412$,
then $\mathcal{B}(\sigma)$ is not planar.

On the other hand, Tenner showed~\cite[Thm. 5.3]{Ten} that if $\sigma$
avoids both $321$ and $3412$, then the interval in Bruhat order
between the identity and $\sigma$ is isomorphic to the Boolean
lattice.  In this case, all the edges in the Bruhat graph come from
covering relations in the Bruhat order, so the Bruhat graph is the
edge graph for a cube of dimension $\ell(\sigma)$.  (One way to see
this is by noting that $\sigma$ avoids $3412$ and $4231$, so its
Bruhat graph has exactly $\ell(\sigma)$ edges at each vertex, and the
covering relations in the Boolean lattice already provide
$\ell(\sigma)$ edges.)  The edge graph of an $n$-dimensional cube is
planar if and only if $n\leq 3$.  Therefore, $\mathcal{B}(\sigma)$ is
planar if and only if $\sigma$ avoids $321$ and $3412$ and
$\ell(\sigma)\leq 3$.  Since $\ell(3412)=4$, the condition that
$\sigma$ avoid $3412$ is encompassed in the condition that
$\ell(\sigma)\leq 3$.
\end{proof}

We also have a longer proof of the above theorem which is independent
of~\cite{Ten}.

The following characterization in terms of the Bruhat graph itself
follows immediately from the above proof.

\begin{cor}
The graph $\mathcal{B}(\sigma)$ is planar if and only if it is
a point, a single edge, or the edge graph of a square or a cube.
\end{cor}

To state our characterization purely in terms of pattern avoidance, we use the following
theorem of Atkinson~\cite[Thm. 2.3]{Atk-fin}.

\begin{thm}
\label{multiplefinite}
Given a permutaton $\pi\in S_k$, the set of permutations containing $\pi$ at most $m$ times is characterized by avoiding a finite set of permutations, all of which are in $S_n$ for some $n\leq k(m+1)$.
\end{thm}

Since the length of a permutation is precisely the number of times it contains the permutation $21$, combining Theorems~\ref{multiplefinite} and~\ref{thm:planar} produces the
following corollary.

\begin{cor}
\label{cor:avoid}
The graph $\mathcal{B}(\sigma)$ is planar if and only if
$\sigma$ avoids $321$ and all permutations $\pi\in S_m$ where $m\leq
8$ and $\ell(\pi)\geq 4$.
\end{cor}

A computer calculation to reduce the set of permutations given by
Corollary~\ref{cor:avoid} to a minimal avoidance set produces the
following.

\begin{cor}
The graph $\mathcal{B}(\sigma)$ is planar if and only if $\sigma$
avoids all of the permutations in the following list:
\noindent\{321, 3412, 23451, 23514, 24153, 25134, 31452, 31524, 41253,
51234, 234165, 231564, 231645, 241365, 214563, 214635, 215364, 216345,
314265, 312564, 312645, 412365, 2315476, 2143675, 2143756, 2145376,
2153476, 3125476, 21436587\}
\end{cor}

Finally, we record the number of permutations with planar Bruhat graphs by
length, which follows from~\cite[Cor. 5.5]{Ten}:

\begin{cor}
For any $m\geq 1$, there are
\begin{enumerate}
\item 1 permutation in $S_m$ of length 0,
\item $(m-1)$ permutations in $S_m$ of length 1,
\item $\frac{(m+1)(m-2)}{2}$ permutations in $S_m$ of length 2, and
\item $\frac{(m+4)(m-1)(m-3)}{6}$ permutations in $S_m$ of length 3
\end{enumerate}
which have planar Bruhat graphs.
\end{cor}

\section{Toroidal Bruhat graphs}

Now we characterize the permutations whose Bruhat graphs are toroidal.

\begin{thm}
\label{thm:toroidal}
Let $\sigma$ be a permutation.  Then the Bruhat graph $\mathcal{B}(\sigma)$ is toroidal if and only if all three of the following conditions hold:
\begin{itemize}
\item $\sigma$ avoids 3412,
\item If $\sigma$ contains 321, then $\ell(\sigma)=3$,
\item $\ell(\sigma)\leq 4$.
\end{itemize}
\end{thm}

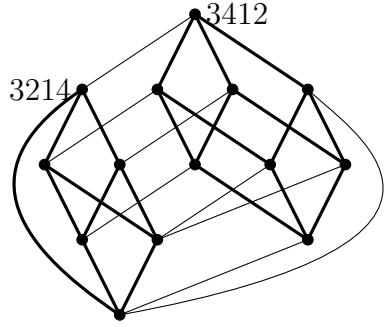
\begin{figure}
\label{figure:B3412}

\begin{center}
\begin{tikzpicture}

\coordinate (1234) at (1,0);
\coordinate (2134) at (0.5,1);
\coordinate (3124) at (0,2);
\coordinate (1324) at (1.5,1);
\coordinate (2314) at (1,2);
\coordinate [label=left:{$3214$}] (3214) at (0.5,3);

\coordinate (1243) at (3.5, 1);
\coordinate (1423) at (4, 2);
\coordinate (1342) at (3, 2);
\coordinate (1432) at (3.5, 3);

\coordinate (2143) at (2, 2);
\coordinate (3142) at (1.5, 3);
\coordinate (2413) at (2.5, 3);
\coordinate [label=right:{$3412$}] (3412) at (2, 4);

\foreach \x in {(1234), (2134), (3124), (1324), (2314), (3214), (1243), (1423), (1342), (1432), (2143), (3142), (2413), (3412)}
  {\filldraw \x circle (2pt);}

\begin{scope}[very thick]

\draw (1234) -- (2134);
\draw (1234) -- (1324);
\draw (2134) -- (3124);
\draw (2134) -- (2314);
\draw (1324) -- (3124);
\draw (1324) -- (2314);
\draw (3124) -- (3214);
\draw (2314) -- (3214);

\draw (1234) .. controls +(-1.5,1) and +(-1.5,-1) .. (3214);

\draw (1243) -- (1423);
\draw (1243) -- (1342);
\draw (1243) -- (2143);
\draw (1423) -- (1432);
\draw (1423) -- (2413);
\draw (1342) -- (1432);
\draw (1342) -- (3142);
\draw (2143) -- (3142);
\draw (2143) -- (2413);
\draw (1432) -- (3412);
\draw (3142) -- (3412);
\draw (2413) -- (3412);

\end{scope}

\draw (1234) -- (1243);
\draw (1234) .. controls +(5.5,1) and +(0.5,-0.5) .. (1432);
\draw (1324) -- (1423);
\draw (1324) -- (1342);
\draw (2134) -- (2143);
\draw (3124) -- (3142);
\draw (2314) -- (2413);
\draw (3214) -- (3412);

\end{tikzpicture}
\end{center}
\caption{The Bruhat graph $\mathcal{B}(3412)$}
\end{figure}

\begin{proof}
First we show that the graph $G=\mathcal{B}(3412)$, drawn in Figure~\ref{figure:B3412} is not toroidal.  Consider the subgraph $H=\mathcal{B}(3214)\subseteq G$ and the induced subgraph $K\subseteq G$ formed by the vertices not in $H$; these subgraphs are drawn with thicker lines in the figure, with $H$ to the left and $K$ to the right.  The subgraph $H$ is isomorphic to the complete bipartite graph $K_{3,3}$, and, in any drawing of $K_{3,3}$ on the torus, each face is homeomorphic to the plane. Also, each vertex in $H$ is connected to at least one vertex of $K$.  The subgraph $K$ is isomorphic to the edge graph of a cube, which cannot be drawn on the plane with all the vertices appearing on a single face.  (In other words, it is not outerplanar.)

Suppose for contradiction we have a drawing of $G$ on a torus with no edges crossing.  If vertices of $K$ are found on two different faces of the drawing of $H$, then there must be a crossing between an edge of $K$ and an edge of $H$.  On the other hand, if all the vertices of $K$ are found on a single face $\Phi$ of the drawing of $H$, then recall that (assuming this drawing of $H$ has no crossings), $\Phi$ is homeomorphic to the plane.  Since $K$ cannot be drawn on the plane with all the vertices appearing on a single face, there must be a vertex $v\in K$ which is not on the unbounded face of $K$ as drawn in $\Phi$.  There is an edge between $v$ and some vertex of $H$, and this edge must cross some edge of $K$.  Hence, if $\sigma$ contains $3412$, $\mathcal{B}(\sigma)$ is not toroidal.

Now we show that, if $\sigma$ contains $321$ and $\ell(\sigma)\geq 4$, then $\mathcal{B}(\sigma)$ is
not toroidal.  We first find a transposition $t$ such that $\ell(\sigma t)<\ell(\sigma)$ and $\sigma t$ also contains $321$.  Fix indices $i_1<i_2<i_3$ such
that $\sigma(i_1)>\sigma(i_2)>\sigma(i_3)$.  Since $\ell(\sigma)\geq 4$, there must be an inversion of
$\sigma$ involving some index $j\not\in\{i_1,i_2,i_3\}$.  If there is such an inversion involving two
indices $j_1,j_2\not\in\{i_1,i_2,i_3\}$, let $t$ be the transposition $t=(j_1\,\, j_2)$.  Otherwise, we divide first into three cases, where $j<i_1$, where $i_1<j<i_3$, and where $i_3<j$.  If $j<i_1$, let $i\in\{i_1,i_2,i_3\}$ be the minimum index such that $\sigma(j)>\sigma(i)$ (so that indices $i$ and $j$ form an inversion), and let $t=(j\,\,i)$.  If $j>i_3$ let $i\in\{i_1,i_2,i_3\}$ be the maximum index such that $\sigma(i)>\sigma(j)$ (so that indices $i$ and $j$ form an inversion), and let $t=(j\,\, i)$.  If $i_1<j<i_3$, we need several further subcases.  If $j$ and $i_2$ form an inversion, then let $t=(j\,\, i_2)$.  Otherwise, if $j<i_2$, then, since $j$ and $i_2$ do not form an inversion, $\sigma(j)<\sigma(i_2)<\sigma(i_1)$, so $j$ and $i_1$ form an inversion, and we let $t=(j\,\,i_1)$.  Finally, if $i_2<j<i_3$ and $\sigma(i_2)<\sigma(j)$, let $t=(j\,\,i_3)$.  Since $t$ corresponds to an inversion of $\sigma$, $\ell(\sigma t)<\ell(\sigma)$.  In the case where $i_1<j<i_3$, the permutation $\sigma t$ contains $321$ at the indices $t(i_1)$, $t(i_2)$, and $t(i_3)$; in the other cases, the permutation $\sigma t$ contains $321$ at the indices $i_1$, $i_2$, and $i_3$.

Given such a transposition $t$, we note that, by Proposition~\ref{patterntograph},
$\mathcal{B}(\sigma)$ includes two disjoint subgraphs isomorphic to $\mathcal{B}(321)$, one with the sink at vertex $\sigma$ and the other with the sink at vertex $\sigma t$.  (In the case where $\sigma$ and $\sigma t$ contain $321$ at the same indices $i_1<i_2<i_3$, the two subgraphs are disjoint since the sets $\{\sigma(i_1), \sigma(i_2),\sigma(i_3)\}$ and $\{\sigma t(i_1), \sigma t(i_2), \sigma t(i_3)\}$ are not equal.  In the last case where $\sigma$ and $\sigma t$ contain $321$ at different indices, the two subgraphs are disjoint since $\{\sigma(i_1),\sigma(i_2),\sigma(i_3)\}$ and $\{\sigma t(t(i_1)),\sigma t(t(i_2)),\sigma t(t(i_3))\}$ are the same.)  Hence $\mathcal{B}(\sigma)$ contains two disjoint subgraphs neither of which are planar.  Since $\mathcal{B}(\sigma)$ is connected, a theorem of Battle, Harary, Kodama, and Youngs~\cite{BHKY} (see also~\cite[Thm. 3.5.3]{GrossTucker}) now implies that $\mathcal{B}(\sigma)$ cannot be toroidal.

If $\sigma$ contains $321$ and $\ell(\sigma)=3$, then $\mathcal{B}(\sigma)$ is isomorphic to $\mathcal{B}(321)$, which is $K_{3,3}$.  This graph is well known to be toroidal.

Finally, due to a theorem of Tenner~\cite{Ten} previously mentioned in the proof of Theorem~\ref{thm:planar}, if $\sigma$ avoids both $321$ and $3412$, then $\mathcal{B}(\sigma)$ is the edge graph of a cube of dimension $\ell(\sigma)$.  The edge graph of the $4$-dimensional cube is toroidal, but the edge graph of the $5$-dimensional cube is not.

\end{proof}

As before, we can give this characterization in terms of the Bruhat graph.

\begin{cor}
The graph $\mathcal{B}(\sigma)$ is toroidal if and only if it is a point, a single edge, the edge graph of a square, a cube, or a $4$-cube, or the complete bipartite graph $K_{3,3}$.
\end{cor}

Using Theorem~\ref{multiplefinite}, we see the following.

\begin{cor}
The graph $\mathcal{B}(\sigma)$ is toroidal if and only if $\sigma$ avoids all permutations $\pi\in S_m$ where $m\leq 10$ and $\mathcal{B}(\pi)$ is not toroidal.
\end{cor}

A computer calculation produces an explicit minimal list of 92 permutations that must be avoided.

The number of permutations with toroidal Bruhat graphs, enumerated by length, is as follows.

\begin{cor}
For an $m\geq1$, there are
\begin{enumerate}
\item $1$ permutation of length 0,
\item $(m-1)$ permutations of length 1,
\item $\frac{(m+1)(m-2)}{2}$ permutations of length 2,
\item $\frac{(m+4)(m-1)(m-3)}{6}+(m-2)$ permutations of length 3, and
\item $\frac{(m+1)(m-4)(m^2+5m-18)}{24}$ permutations of length 4
\end{enumerate}
which have toroidal Bruhat graphs
\end{cor}

\begin{proof}
A permutation $\sigma\in S_m$ which contains $321$ and has length $3$ must be a transposition of the form $(i\,\,i+2)$ for some $i$ with $1\leq i\leq m-2$.  The rest of the enumeration follows from~\cite[Cor 5.5]{Ten}, as in the case of planar Bruhat graphs.
\end{proof}

\end{document}